\documentclass[12pt]{amsart}
\numberwithin{equation}{section}
\usepackage[utf8]{inputenc}
\usepackage{fullpage}
\usepackage{amsmath,xcolor}
\usepackage{amssymb,mathrsfs}
\usepackage{newtxtext,newtxmath}
\usepackage{amsthm}
\usepackage{bm}
\usepackage{bbm}
\usepackage{hyperref}
\let\originalleft\left
\let\originalright\right
\renewcommand{\left}{\mathopen{}\mathclose\bgroup\originalleft}
\renewcommand{\right}{\aftergroup\egroup\originalright}

\newcommand{\N}{\mathbb{N}}
\newcommand{\R}{\mathbb{R}}
\newcommand{\Z}{\mathbb{Z}}
\newcommand{\E}{\mathbb{E}}
\renewcommand{\P}{\mathbb{P}}

\newcommand{\dd}{\text{d}}

\renewcommand{\P}{\mathbb{P}}

\newtheorem{theorem}{Theorem}[section]
\newtheorem*{theorem*}{Theorem}
\newtheorem{lemma}[theorem]{Lemma}
\newtheorem{proposition}[theorem]{Proposition}
\newtheorem*{proposition*}{Proposition}

\newtheorem{condition}{Condition}

\theoremstyle{definition}
\newtheorem{remark}[theorem]{Remark}

\newcommand{\relmiddle}[1]{\mathrel{}\middle#1\mathrel{}}
\newcommand{\cond}{\relmiddle{|}}

\newcommand{\1}{\mathbbm 1}

\newcommand{\p}{{p^*}}

\begin{document}

\title{Moment characterization of the weak disorder phase for directed polymers in a class of unbounded environments}
\author[R.~Fukushima]{Ryoki Fukushima}
\address{Department of Mathematics, University of Tsukuba, 1-1-1 Tennodai, Tsukuba, Ibaraki 305--8571, Japan}
\email{ryoki@math.tsukuba.ac.jp}

\author[S.~Junk]{Stefan Junk}
\address{Advanced Institute for Materials Research Tohoku University Mathematical Group, 2-1-1 Katahira, Aoba-ku, Sendai, 980-8577 Japan}
\email{sjunk@tohoku.ac.jp}

\begin{abstract}
  For a directed polymer model in random environment, a characterization of the weak disorder phase in terms of the moment of the renormalized partition function has been proved in [S.~Junk: Communications in Mathematical Physics 389, 1087--1097 (2022)]. We extend this characterization to a large class of unbounded environments which includes many commonly used distributions.
\end{abstract}

\maketitle

\section{Introduction}
\label{sec:intro}
We consider a model of directed polymer in random environment. Let $(X=(X_j)_{j\ge 0}, P^{\text{SRW}})$ be the simple random walk on $\Z^d$ starting at the origin and $((\omega_{j,x})_{(j,x)\in \N\times\Z^d},\P)$ be a sequence of independent and identically distributed random variables satisfying
\begin{align}\label{eq:expmom}
	e^{\lambda(\beta)}:=\E\big[e^{\beta \omega_{0,0}}\big]<\infty\text{ for all $\beta\ge 0$.}
\end{align}

Then we define the law of the polymer of length $n$ at the inverse temperature $\beta \ge 0$ by
\begin{equation}
  \label{eq:polymer}
  \dd \mu_{\omega,n}^\beta(\dd X)=\frac{1}{Z_n^\beta(\omega)}\exp\Big(\beta\sum_{j=1}^n \omega_{j,X_j}\Big)P^\text{SRW}(\dd X),
\end{equation}
where $Z_n^\beta(\omega)=E^{\text{SRW}}[\exp(\beta\sum_{j=1}^n \omega_{j,X_j})]$ is the normalizing constant, called the \emph{partition function} of the model. Under this measure, the random walk is attracted by the sites where $\omega$ is positive, and repelled by the sites where it is negative. Thus we expect that the behavior of the polymer is strongly affected by the environment when $\beta$ is large. 

This intuition is made precise in~\cite{CSY03,CY06} under the assumption $\E\big[e^{\beta \omega_{0,0}}\big]<\infty$ for all $\beta\in \R$. In spatial dimension $d\geq 3$, there exists $\beta_{cr}\in(0,\infty)$ such that for $\beta<\beta_{cr}$, 
\begin{align}
  \label{eq:WD}
  e^{-n\lambda(\beta)}Z_n^\beta(\omega) \xrightarrow{n\to \infty}W_\infty^\beta(\omega)>0, \quad \P\text{-a.s.},
\end{align}
whereas for $\beta>\beta_{cr}$, 
\begin{align}
  \label{eq:SD}
  e^{-n\lambda(\beta)}Z_n^\beta(\omega) \xrightarrow{n\to \infty}0, \quad \P\text{-a.s.}
\end{align}
As one can readily verify that the annealed partition function satisfies $E[Z_n^\beta]=e^{n\lambda(\beta)}$, the above shows that the quenched and annealed partition functions are comparable for $\beta<\beta_{cr}$ and contrary for $\beta>\beta_{cr}$. This indicates that the effect of disorder is weak in the former phase and strong in the latter phase with a drastic change in behavior across $\beta_{cr}$. We refer the interested reader to~\cite{CSY03,CY06}. 

\smallskip The proof of the aforementioned results relies on the fact that $W_n^\beta(\omega):=e^{-n\lambda(\beta)}Z_n^\beta(\omega)$ is a non-negative martingale under $\P$ with the filtration $\mathcal{F}_n:=\sigma(\omega_{j,x}\colon j\le n, x\in\Z^d)$, and one can further show  that the phase~\eqref{eq:WD} is characterized by the unform integrability of $W_n^\beta(\omega)$. But in order to further analyze the weak disorder phase, it is desirable to have a stronger property for $(W_n^\beta(\omega))_{n\ge 0}$. The second author has recently proved in~\cite{J21_1} that for $\beta<\beta_{cr}$, the martingale $(W_n^\beta(\omega))_{n\ge 0}$ is $L^p$-bounded for some $p>1$, under the assumption that the random potential $\omega$ is bounded from above. The main result of this paper extends this characterization to a large class of unbounded environments. 

\section{Main result}
\label{sec:result}

We introduce the following condition for the environment $\omega$.
\begin{condition}\label{cond:1}
For $\beta>0$, there exist $A_1=A_1(\beta)>1$ and $c_1=c_1(\beta)>0$ such that, for all $A>A_1$,
\begin{align}\label{eq:cond1}
	E\left[\left.e^{\beta\omega}\;\right|\;\omega>A\right]\leq c_1e^{\beta A}.
\end{align}
\end{condition}
This condition strengthens the assumption \eqref{eq:expmom} of finite exponential moments by requiring a control on the overshoot when $\omega$ is conditioned to be large. It does not seem to be very restrictive and holds for many commonly used distributions, although we stress that there are distributions that satisfy \eqref{eq:expmom} but not Condition \ref{cond:1}. We elaborate on these matters in Section~\ref{sec:cond}.

\smallskip The following is the main result of this paper. 
\begin{theorem}
  \label{thm:main}
  Let $\beta$ be such that $\P(W_\infty^\beta>0)>0$ and assume that $\omega$ satisfies \eqref{eq:expmom} and  Condition \ref{cond:1}.Then there exists $p=p(\beta)>1$ such that 
  \begin{align}
    \label{eq:L^p-bdd}
    \sup_{n\in \N}\|W_n^\beta\|_p<\infty. 
  \end{align}
Moreover, the set of $p>1$ such that \eqref{eq:L^p-bdd} holds is open.
\end{theorem}

\begin{remark}
If $\lim_{n\to\infty}W_n^\beta=0$, then $W_n^\beta$ is not uniformly integrable and hence~\eqref{eq:L^p-bdd} necessarily fails. Thus the weak disorder is characterized by the finiteness of a $p$-th moment. 
\end{remark}

\begin{remark} 
	In \cite{J21_1}, it was further shown that if $\omega$ is bounded from below, then $\sup_n\E[W_n^{-\varepsilon}]<\infty$ for some $\varepsilon>0$. The argument in this paper can easily be generalized to show that the same holds whenever $\omega$ satisfies the straightforward generalization of Condition \ref{cond:1} to the negative tail.
\end{remark} 

\begin{remark}
It is an interesting problem to describe the dependence of the optimal exponent $p^*(\beta):=\sup\{p\colon (W_n^\beta)_{n\in\N}\text{ is $L^p$ bounded}\}$ as a function of $\beta$. For bounded environments, it has been shown in \cite{J22} that $\p(\beta)\geq 1+2/d$ whenever $W_\infty^\beta>0$, so that $\beta\mapsto p^*(\beta)$ has a discontinuity at $\beta_{cr}$. It is natural to expect that the same holds in general. 
\end{remark}
\section{Extension of Condition~\ref{cond:1}}

As will be explained in detail below, the main step in proving Theorem \ref{thm:main} is to control the overshoot of $W_\tau$ at a stopping time $\tau$, which takes the form
\begin{align}\label{eq:convex_comb}
	\frac{W_{\tau}^\beta}{W_{\tau-1}^\beta}=	\sum_x \alpha_x e^{\beta\omega_x-\lambda(\beta)}
\end{align}
for a certain choice of probability weights $(\alpha_x)_{x\in\Z^d}$. The purpose of this section is to translate the Condition \ref{cond:1} on $\omega$ into a statement about such convex combinations.

\smallskip First, we state a condition satisfied by $e^{\beta\omega-\lambda(\beta)}$ whenever $\omega$ satisfies Condition \ref{cond:1}. 
\begin{condition}
  \label{cond:2}
The random variable $Y$ is non-negative with $\E[Y]=1$, $\E[Y^2]<\infty$ and there exist $A_2>1$ and $c_2>0$ such that, for all $p\in [1,2]$ and $A\ge A_2$, 
\begin{align}
\label{eq:cond2}
  \E\left[Y^p \cond Y> A\right] \le c_2A^p.
\end{align}
\end{condition}

The next condition requires additionally that \eqref{eq:cond2} extends to convex combinations.
\begin{condition}\label{cond:3}
The random variable $Y$ is non-negative with $\E[Y]=1$, $\E[Y^2]<\infty$ and there exist $A_3>1$ and $c_3>0$ such that the following holds: If $(Y_i)_{i\in I}$ are i.i.d. copies of $Y$ and $(\alpha_{i})_{i \in I}$ is a collection non-negative numbers with $\sum_{i\in I}\alpha_i=1$, then for all $p\in[1,2]$ and $A\geq A_3$
\begin{align}\label{eq:cond3}
	\E\left[\Big(\textstyle{\sum_{i\in I}\,} \alpha_i Y_i\Big)^p \cond \textstyle{\sum_{i\in I}\,} \alpha_i Y_i > A\right] \le c_3A^p.
\end{align}
\end{condition}

We now show that both conditions follow from Condition \ref{cond:1}.

\begin{lemma}\label{lem:claim2}
\begin{enumerate}
\item [(i)] If $\omega$ satisfies Condition \ref{cond:1}, then $Y:=e^{\beta\omega-\lambda(\beta)}$ satisfies Condition \ref{cond:2}.
\item [(ii)] If a random variable $Y$ satisfies Condition \ref{cond:2}, then it also satisfies Condition \ref{cond:3}.
\end{enumerate}
\end{lemma}

\begin{proof}
The proof of \textbf{part (i)} is simple. 
For $A\geq A_2:=e^{\beta A_1(2\beta)-\lambda(\beta)}$, we can use Condition~\ref{cond:1} to get
\begin{align*}
\E[Y^2|Y>A]&=\E\Big[e^{2\beta\omega}\;\Big|\;\omega>\frac{1}{\beta}(\log A+\lambda(\beta))\Big]e^{-2\lambda(\beta)}\\
	   &\leq c_0(2\beta) e^{2\beta \frac{1}{\beta}(\log A+\lambda(\beta))}e^{-2\lambda(\beta)}\\
	   &=:c_1 A^2. 
\end{align*}
The extension to $p\in[1,2)$ follows from Jensen's inequality. 

\smallskip The proof of \textbf{part (ii)} is more involved. In the following, we use $C$ for positive constants depending only on $\E[Y_i^2]$, $A_2$ and $c_2$, whose values may change from line to line.  Let $A\ge A_3:=A_2$ and $N:=\sum_i \1_{\{\alpha_iY_i>A\}}$. We separately consider the case where all the summands are small ($N=0$) and the cases where the event $\sum_i\alpha_iY_i>A$ is realized due to a single large summand ($N\geq 1$). In the first case, 
we have 
\begin{equation}\label{eq:case1}
	\E\Big[\Big(\textstyle{\sum_{i}\,} \alpha_i Y_i\Big)^2\1_{\{N=0\}}\;\Big|\; \textstyle{\sum_i\,} \alpha_i Y_i>A\Big]
	\leq \E\Big[\Big(\textstyle{\sum_i\,} \alpha_i Y_i\1_{\{\alpha_iY_i\leq A\}}\Big)^2\;\Big|\; \textstyle{\sum_{i}\,} \alpha_i Y_i>A\Big]
\end{equation}
since $Y_i=Y_i\1_{\{\alpha_iY_i\leq A\}}$ for all $i$ on $\{N=0\}$. Let $\tau:=\inf\{i\colon \sum_{j\leq i}\alpha_jY_j>A\}$ and observe that on $\{\sum_i \alpha_i Y_i>A\}=\{\tau<\infty\}$,
\begin{align*}
\sum_{i\leq \tau}\alpha_iY_i\1_{\{\alpha_iY_i\leq A\}}\leq \sum_{i<\tau}\alpha_iY_i+\alpha_\tau Y_{\tau}\1_{\{\alpha_\tau Y_\tau \leq A\}}\leq 2A.
\end{align*}
Note also that conditioned on $\tau=i$, the remaining variables $(Y_{j+i})_{j\ge 1}$ obey the unconditioned law $\P$. Therefore, 
\begin{equation}
  \label{eq:case1-bound}
\begin{split}
  \E\left[\Big(\textstyle{\sum_{i}\,} \alpha_i Y_i\Big)^2\1_{\{N=0\}}\cond \textstyle{\sum_{i}\,} \alpha_i Y_i>A\right]
  &\leq \E\left[\Big(2A+\textstyle{\sum_{i>\tau}\,}\alpha_i Y_i\Big)^2\cond \tau<\infty\right]\\
  &\leq \E\left[\Big(2A+\textstyle{\sum_{i\in I}\,} \alpha_i Y_i\Big)^2\right]\\
  &\leq C(A^2+1),
\end{split}
\end{equation}
where in the last line, we have used $\sum_{i\in I}\alpha_i=1$ and that $Y_1$ has a finite second moment. In the second case $N\ge 1$, we use $\{N\geq 1\}\subseteq\{\sum_i\alpha_iY_i>A\}$ to obtain
\begin{equation}\label{eq:case2}
  \E\Big[\Big(\textstyle{\sum_{i}\,} \alpha_i Y_i\Big)^2\1_{\{N\geq 1\}}\;\Big|\; \textstyle{\sum_{i}\,} \alpha_i Y_i>A\Big]
  \leq \E\Big[\Big(\textstyle{\sum_{i}\,} \alpha_i Y_i\Big)^2\;\Big|\; N\geq 1\Big].
  \end{equation}
Let $q_i:=\P(\alpha_iY_i>A\mid N\geq 1)$ and observe that 
\begin{equation}
  \label{eq:square}
\begin{split}
\alpha_i^2\E[Y_i^2\mid N\geq 1]
&\le \alpha_i^2\E[Y_i^2\mid \alpha_iY_i>A]q_i+\alpha_i^2\E[Y_i^2\mid \alpha_iY_i\leq A]\\
&\leq c_2A^2q_i+\alpha_i^2\E[Y_i^2],
\end{split}
\end{equation}
where we  have used Condition~\ref{cond:2} for the first term (note that $A/\alpha_i\geq A_2$)  and the negative correlation between $Y_i^2$ and $\1_{\{\alpha_iY_i\leq A\}}$ for the second term. 

Similarly, for $i\neq j$, let $q_{i,j}:=\P(\alpha_iY_i>A,\alpha_jY_j>A\mid N\geq 1)$ and observe that 
\begin{equation}
  \label{eq:cross}
\begin{split}
\alpha_i\alpha_j\E[Y_iY_j\mid N\geq 1]&\leq \alpha_i\alpha_j\big(q_{i,j}\E[Y_i\mid \alpha_i Y_i>A]\E[Y_j\mid \alpha_jY_j>A]\\
&\quad+q_i\E[Y_i\mid \alpha_i Y_i>A]+q_j\E[Y_j\mid \alpha_jY_j>A]\\
&\quad+\E[Y_iY_j\1_{\{\alpha_i Y_i\le A\}}\1_{\{\alpha_j Y_j\le A\}}\mid \textstyle{\max_{k\neq i,j}\alpha_kY_k}> A]\big)\\
&\leq C\big(q_{i,j} A^2+\alpha_jq_i A+\alpha_iq_j A+\alpha_i\alpha_j\big),
\end{split}
\end{equation}
where we have used Condition~\ref{cond:1} for the first three terms and that $Y_iY_j\1_{\{\alpha_i Y_i\le A\}}\1_{\{\alpha_j Y_j\le A\}}$ is independent of $\{\max_{k\neq i,j}\alpha_kY_k> A\}$ for the last term. To bound the right-hand side of \eqref{eq:case2}, we are going to sum~\eqref{eq:square} over $i$ and~\eqref{eq:cross} over $i\neq j$. Note that $\sum_{i}q_i=\E[N\mid N\geq 1]$ and $\sum_{i\neq j}q_{i,j}\le \E[N^2\mid N\geq 1]$. We have the following bounds on these quantities. 
\begin{lemma}  \label{lem:N}
In the above setup, it holds that $\E[N\mid N\geq 1]\leq 2$ and $\E[N^2\mid N\geq 1]\leq 5$. 
\end{lemma}
The proof of this lemma will be given below. Now, summing~\eqref{eq:square} over $i$ and~\eqref{eq:cross} over $i\neq j$ and then using Lemma~\ref{lem:N}, it follows that the left-hand side of~\eqref{eq:case2} is bounded by $C(A^2+1)$. Combining this with~\eqref{eq:case1-bound} and recalling $A\ge 1$, we get

\begin{align*}
  \E\Big[\Big(\textstyle{\sum_{i}\,} \alpha_i Y_i\Big)^2 \;\Big|\; \textstyle{\sum_{i}\,} \alpha_i Y_i > A\Big] \le C (A^2+1)\leq 2CA^2.
\end{align*}

Finally, the claim for $p\in[1,2)$ follows as before by applying Jensen's inequality to the above. 
\end{proof}

\begin{proof}[Proof of Lemma \ref{lem:N}]
	Let $\sigma=\inf\{i \colon \alpha_i Y_i >A\}$ and write $N=\1_{\sigma<\infty}+\sum_{i>\sigma} \1_{\{\alpha_iY_i>A\}}$. Conditioned on $\sigma=i$, the random variables $(Y_{i+j})_{j\geq 1}$ obey the unconditioned law $\P$. Therefore,
  \begin{align*}
  \E[N\mid N\geq 1]&=\E[N\mid \sigma<\infty]\leq 1+\E[N],\\
  \E[N^2\mid N\geq 1]&=\E[N^2\mid \sigma<\infty]\leq \E[(1+N)^2],
  \end{align*}
  and hence it suffices to prove that $\E[N]\leq 1$ and $\E[N^2]\leq 2$. Both follow from the Markov inequality:
  \begin{align*}
  \E[N]&=\sum_i\P(Y_i>A/\alpha_i)\leq \E[Y_1]\sum_{i}\frac{\alpha_i}A=\frac{\E[Y_1]}{A},\\
  \E[N^2]&=\sum_i\P(Y_i>A/\alpha_k)+\sum_{i\neq j}\P(Y_i>A/\alpha_i)\P(Y_j>A/\alpha_j)\\
  &\leq \frac{\E[Y_1]}{A}+\frac{\E[Y_1]^2}{A^2}.
  \end{align*}
Recalling $\E[Y_1]=1$ and $A\geq 1$, this implies the desired bounds.
\end{proof}

\section{Proof of Theorem~\ref{thm:main}}
In this section, we prove Theorem~\ref{thm:main}. Let us start by recalling the relevant steps of the proof of~\cite[Theorem 1.1 (ii)]{J21_1}. For $t>1$, define the stopping time
\begin{equation}
  \tau(t):=\inf\left\{n\in \N\colon W_n^\beta \ge t\right\}
\end{equation}
and the pinned version of  $W_n^\beta$ as follows:
\begin{equation}
  W_{n,x}^\beta:=E^\text{SRW}\left[\exp\left(\sum_{t=1}^n (\beta\omega_{t,X_t}-\lambda(\beta))\right); X_n=x\right]
\end{equation}
Then, by using the Markov property for the simple random walk, we write on $\{\tau(t)\leq n\}$
\begin{equation}
  \label{eq:stopping}
\begin{split}
  W_n^\beta&=\sum_{x\in\Z^d} W_{\tau(t),x}^\beta \left(W_{n-\tau(t)}^\beta \circ \theta_{\tau(t),x}\right)\\
  &=W_{\tau(t)}^\beta\sum_{x\in\Z^d} \mu_{\omega,\tau(t)}^\beta(X_{\tau(t)}=x) \left(W_{n-\tau(t)}^\beta \circ \theta_{\tau(t),x}\right),
\end{split}
\end{equation}
where $\theta_{k,x}$ stands for the time-space shift of the environment. By \eqref{eq:stopping} and Jensen's inequality, we have 
\begin{align}\label{eq:argument}
  \E\left[\big(W_n^\beta\big)^p \1_{\{\tau(t)=k\}}\right]
  \le \E\left[\big(W_k^\beta\big)^p \1_{\{\tau(t)=k\}}\right]\E\left[\big(W_{n-k}^\beta\big)^p\right].
\end{align}
To continue the argument, we need the following bound on the second factor, uniformly in $t>1$ and $p\in[1,2]$: 
\begin{align}
\label{eq:need}
\E\left[\big(W_k^\beta\big)^p \1_{\{\tau(t)=k\}}\right]\leq Ct^p\P(\tau(t)=k).
\end{align}
In \cite{J21_1}, the assumption $\omega_{t,x} \le K$ was used to ensure that $(W_{k}^\beta)^p \le e^{2 \beta K}t^p$ on $\{\tau(t)=k\}$, that is, the martingale does not overshoot much at the stopping time $\tau(t)$. 

\smallskip We replace this part of the argument by using Lemma~\ref{lem:claim2}. Let $c_3$ and $A_3$ be the constants obtained by applying Lemma~\ref{lem:claim2}(i)--(ii).  We now bound the left-hand side in \eqref{eq:need} by considering the cases $W_k^\beta\le A_3t$ and $W_k^\beta>A_3t$ separately. The first case is simple:
\begin{equation}
  \label{eq:firstcase}
  \E\left[\big(W_k^\beta\big)^p \1_{\{\tau(t)=k,W_k^\beta\le A_3t\}}\right]\le (A_3t)^p\P\big(\tau(t)=k,W_k^\beta\le A_3t\big).
\end{equation}
In the second case, we consider the conditional expectation given $\mathcal{F}_{k-1}$ to write
\begin{equation}
  \label{eq:stop@k}
  \E\left[\big(W_k^\beta\big)^p \1_{\{\tau(t)=k,W_k^\beta> A_3t\}}\right]
  =\E\left[\big(W_{k-1}^\beta\big)^p \1_{\{\tau(t)>k-1\}}\E\left[\left({W_k^\beta}/{W_{k-1}^\beta}\right)^p \1_{\{W_k^\beta>A_3t\}}\cond \mathcal{F}_{k-1}\right]\right].
\end{equation}
We further rewrite\footnote{In the following equation, we regard $\mu_{\omega,k-1}^\beta$ as a measure on the space of \emph{infinite} path while the interaction with the environment is restricted to time interval $[0,k-1]$.} 
\begin{align*}
{W_k^\beta}/{W_{k-1}^\beta}=\textstyle{\sum_{x}\,} \alpha_x Y_x\text{ and }\big\{W_k^\beta>A_3t\big\}=\big\{\textstyle{\sum_{x}\,}\alpha_xY_x>A\big\},
\end{align*}
where $\alpha_x:=\mu_{\omega,k-1}^\beta(X_k=x)$, $Y_x:=e^{\beta\omega_{k,x}-\lambda(\beta)}$ and $A:=A_3t/W_{k-1}^\beta$. Then, noting that 
\begin{itemize}
  \item $(e^{\beta\omega_{k,x}-\lambda(\beta)})_{x\in\Z^d}$ is independent of $\mathcal{F}_{k-1}$,
  \item $\mu_{\omega,k-1}^\beta(X_k=x)$ is an $\mathcal{F}_{k-1}$-measurable probability measure on $\Z^d$ and 
  \item $t/W_{k-1}^\beta \ge 1$ on $\{\tau(t)>k-1\}$,
\end{itemize}
we can apply Lemma~\ref{lem:claim2} under $\P(\cdot\mid\mathcal{F}_{k-1})$ to obtain
\begin{align*}
  \E\left[\big({W_k^\beta}/{W_{k-1}^\beta}\big)^p \1_{\{W_k^\beta>A_3t\}}\cond \mathcal{F}_{k-1}\right]
  \le c_3\big(A_3t/W_{k-1}^\beta\big)^p\P\left(W_k^\beta/W_{k-1}^\beta> A_3t/W_{k-1}^\beta\cond \mathcal{F}_{k-1}\right).
\end{align*}
Substituting this into~\eqref{eq:stop@k} yields
\begin{equation}
  \label{eq:secondcase}
  \E\left[\big(W_k^\beta\big)^p \1_{\{\tau(t)=k,W_k^\beta>A_3t\}}\right]
  \le c_3A_3^2t^p \P(\tau(t)=k,W_k^\beta>A_3t). 
\end{equation}
Combining this bound with \eqref{eq:firstcase}, we obtain \eqref{eq:need}  and can thus repeat the argument in~\cite[eq.(20)]{J21_1}. 
Since the other parts of the proof of~\cite[Theorem~2.1 (ii)]{J21_1} do not rely on the boundedness assumption, the same argument proves Theorem~\ref{thm:main}.

\section{Discussion on Condition~\ref{cond:1}}
\label{sec:cond}
In this section, we discuss Condition~\ref{cond:1}. First, although it looks natural, it does not hold in general. For example, if $\omega$ is supported on $\{k^2\}_{k\in \N}$, then regardless the concrete form of the distribution of $\omega$, we have 
\begin{align*}
\E\left[e^{\beta\omega} \cond \omega> k^2\right]&=\E\left[e^{\beta\omega} \cond \omega \ge (k+1)^2\right]\geq e^{\beta(k+1)^2} 
\end{align*}
and hence Condition~\ref{cond:1} fails. 

Next, we see that Condition \ref{cond:1} is valid under a one-sided tail regularity assumption, which holds under certain upper and lower bounds on the tail. 
\begin{proposition}\label{prop:sufficient}
Let $\omega$ be a real-valued random variable. 
\begin{enumerate}
	\item [(i)] Assume that there exist $K>0$ and $M>2\beta$ such that
\begin{align}\label{eq:dom-var2}
	\limsup_{x\to\infty}\sup_{y\geq K}\frac{\P(\omega>x+y)e^{My}}{\P(\omega>x)}<\infty.
\end{align}
Then Condition \ref{cond:1} holds.
 \item[(ii)] Assume that there exist $c>0$ and a convex function $f$ satisfying $\lim_{x\to\infty}\frac{f(x)}x=\infty$ such that, for $x$ large enough,
 \begin{align}\label{eq:convex}
	 c^{-1}e^{-f(x)}\leq \P(\omega > x)\leq ce^{-f(x)}.
 \end{align}
Then Condition \ref{cond:1} holds for all values of $\beta$.
\item[(iii)] Assume that there exist $c>0$ and an increasing 
function $f$ satisfying $f(x+y)\geq f(x)f(y)$ such that, for $x$ large enough,
\begin{align}\label{eq:convex2}
	c^{-1}e^{-cf(x)}\leq \P(\omega > x)\leq ce^{-f(x)/c}.
\end{align}
Then Condition \ref{cond:1} holds for all values of $\beta$.
\end{enumerate}
\end{proposition}
This proposition covers many commonly used distributions. 
\begin{itemize}
	\item If $\omega$ has a logarithmically concave Lebesgues density, then $x\mapsto \P(\omega>x)$ is also logarithmically concave (see \cite[Theorem 2]{P71}) and hence \eqref{eq:convex} holds with $f(x):=-\log \P(\omega>x)$. Note also that $\lim_{x\to\infty}\frac{f(x)}{x}=\infty$ already follows from \eqref{eq:expmom}. This covers, for example, the Gaussian distribution or the Weibull distribution (with $\P(\omega>x)=ce^{-c'x^\alpha}$ for $\alpha> 1$).
	\item For the Poisson distribution, it is not hard to check \eqref{eq:dom-var2} directly.
	\item The (negative) Gumbel distribution, with $\P(\omega> x)=\exp({-e^{(x-c)/c'}})$, further satisfies \eqref{eq:convex2}. More generally, we can take $f(x)=e^{x^\alpha}$ with $\alpha\ge 1$ in \eqref{eq:convex2}. 
\end{itemize}

\begin{proof}
	\textbf{Part (i)}: By \eqref{eq:dom-var2}, there exist $K>0$, $M>2\beta$, $A_0>1$ and $C>0$ such that, for $y\geq K$, $A>A_1$ and $u>K+A$,
\begin{align*}
\P(\omega>u)\leq C\P(\omega>A)e^{-Mu+MA}.
\end{align*}
Thus, for $A\geq A_1$,
\begin{align*}
	\E[e^{2\beta\omega}\1_{\omega>A}]&\leq \P(\omega>A)e^{2\beta (A+K)}+\E[e^{2\beta\omega}\1_{\omega>A+K}]\\
					 &=\P(\omega>A)e^{2\beta (A+K)}+\int_{e^{2\beta (A+K)}}^\infty \P(\omega>\log(t)/2\beta)\dd t\\
					 &\leq \P(\omega>A) \Big(e^{2\beta (A+K)}+Ce^{MA}\int_{e^{2\beta (A+K)}}^\infty e^{-M\log(t)/2\beta} \dd t\Big)\\
					 &=\P(\omega>A) \Big(e^{2\beta (A+K)}+\frac{C}{M/2\beta-1}e^{MA}e^{2\beta(A+K)(1-M/2\beta)}\Big)\\
					 &=: c_3\P(\omega>A)e^{2\beta A},
\end{align*}
where we have used  the assumption $M>2\beta$ to ensure the convergence of the last intergral.

\smallskip For \textbf{part (ii)}, it is now enough to verify \eqref{eq:dom-var2}. The convexity and the assumption on superlinear growth imply that there exists $x_0>0$ such that the right derivative $D_+f(x_0) \ge 3\beta$. Then for $x\geq x_0$ and $ y>0$, we have $f(x+y)-f(x) \ge 3\beta y$ and hence 
\begin{align*}
	\frac{\P(\omega>x+y)}{\P(\omega>x)}\leq c^2e^{-(f(x+y)-f(x))}
  \leq c^2 e^{-3\beta y}.
\end{align*}
This implies \eqref{eq:dom-var2}.

For \textbf{part (iii)}, note that by the super-additive theorem there exists $C>0$ such that $f(x)\geq e^{Cx}$, hence for $y>2\log(c)/C$ and $x$ large enough,
\begin{align*}
	\frac{\P(\omega>x+y)}{\P(\omega>x)}\leq c^2\exp\Big(-f(x)\Big(\frac{f(x+y)}{cf(x)}-c\Big)\Big)\leq c^2e^{-f(x)(f(y)/c-c)}\leq c^2 e^{-3\beta y}.
\end{align*}
This again implies \eqref{eq:dom-var2} and we are done.
\end{proof}

\begin{remark}
In Section~\ref{sec:cond}, we rephrased Condition~\ref{cond:1} in terms of the random variable $Y:=e^{\beta\omega-\lambda(\beta)}$. Since some authors use this $Y$ as the random potential in the directed polymer model (see, for example, \cite{IS88,Sep12}), it might be of interest to rephrase also \eqref{eq:dom-var2}, which reads
\begin{align}\label{eq:dom-var}
	\text{there exist } K>1\text{ and } M>2\text{ such that }\limsup_{y\to \infty}\sup_{\lambda\ge K}\lambda^{M}\frac{\P(Y > \lambda y)}{\P(Y> y)} <\infty.
\end{align}
This is a one-sided regular variation condition. It appears, for example, in~\cite[Theorem~2.0.1]{BGT} and inspecting its proof, one can see that \eqref{eq:dom-var} follows from
\begin{align}
	\text{there exist } K>1, M>2 \text{ and } \rho<K^{-M} \text{ such that }\sup_{\lambda\in[K,K^2]}\limsup_{y\to\infty}\frac{\P(Y > \lambda y)}{\P(Y> y)}<\rho.
\end{align}
There are plenty of distributions that satisfy~\eqref{eq:dom-var}. For instance, if there exist $c,C>0$ and $\gamma>0$ such that 
\begin{align}
  \label{eq:Weibull}
  c\exp(-C y^\gamma)\le \P(Y> y)\le C\exp(-c y^\gamma)
\end{align}
holds for all sufficiently large $y$, then 
\begin{align*}
  \frac{\P(Y > \lambda y)}{\P(Y> y)} 
  &\le \frac{C}{c}\exp(-(c\lambda^\gamma-C) y^\gamma),
\end{align*}
and~\eqref{eq:dom-var} follows. A similar argument applies to the case where $y^\gamma$ in~\eqref{eq:Weibull} is replaced by $\exp(y^\gamma)$ ($\gamma>0$) or $\exp(\log^\alpha y)$ $(\alpha>1$). 
\end{remark}

\section*{Acknowledgment} 
This work was supported by KAKENHI 21K03286, 22H00099 and 18H03672.  

\bibliographystyle{plain}
\bibliography{ref}

\begin{thebibliography}{1}

\bibitem{BGT}
N.~H. Bingham, C.~M. Goldie, and J.~L. Teugels.
\newblock {\em Regular variation}, volume~27 of {\em Encyclopedia of
  Mathematics and its Applications}.
\newblock Cambridge University Press, Cambridge, 1987.

\bibitem{CSY03}
Francis Comets, Tokuzo Shiga, and Nobuo Yoshida.
\newblock Directed polymers in a random environment: path localization and
  strong disorder.
\newblock {\em Bernoulli}, 9(4):705--723, 2003.

\bibitem{CY06}
Francis Comets and Nobuo Yoshida.
\newblock Directed polymers in random environment are diffusive at weak
  disorder.
\newblock {\em Ann. Probab.}, 34(5):1746--1770, 2006.

\bibitem{IS88}
John~Z. Imbrie and Thomas Spencer.
\newblock Diffusion of directed polymers in a random environment.
\newblock {\em J. Statist. Phys.}, 52(3-4):609--626, 1988.

\bibitem{J22}
Stefan Junk.
\newblock Fluctuations of partition functions of directed polymers in weak
  disorder beyond the ${L}^2$-phase.
\newblock \emph{arXiv preprint} arXiv:2202.02907, 2022.

\bibitem{J21_1}
Stefan Junk.
\newblock New {C}haracterization of the {W}eak {D}isorder {P}hase of {D}irected
  {P}olymers in {B}ounded {R}andom {E}nvironments.
\newblock {\em Comm. Math. Phys.}, 389(2):1087--1097, 2022.

\bibitem{P71}
Andr{\'a}s Pr{\'e}kopa.
\newblock Logarithmic concave measures with application to stochastic
  programming.
\newblock {\em Acta Scientiarum Mathematicarum}, 32:301--316, 1971.

\bibitem{Sep12}
Timo Sepp\"{a}l\"{a}inen.
\newblock Scaling for a one-dimensional directed polymer with boundary
  conditions.
\newblock {\em Ann. Probab.}, 40(1):19--73, 2012.

\end{thebibliography}

\end{document}